\documentclass[a4paper,12pt,centertags,oneside]{amsart}
\usepackage{amsmath,amstext,amsthm,a4,amssymb,amscd}
\usepackage[mathscr]{eucal}
\usepackage{mathrsfs}
\usepackage{bbold}
\usepackage{epsf}
\usepackage{typearea}
\usepackage{charter}
\usepackage{bm}

\usepackage{graphicx}
\usepackage[all]{xy}

\usepackage{hyperref}
\hypersetup{
    colorlinks = true,
    linkcolor = black,
    citecolor = black,
    urlcolor = black,
}

\newcommand{\Q}{\mathbb{Q}}
\newcommand{\R}{\mathbb{R}}
\newcommand{\C}{\mathbb{C}}

\newcommand{\N}{\mathbb{N}}

\makeatletter
\newcommand*{\rom}[1]{\expandafter\@slowromancap\romannumeral #1@}
\makeatother

\newtheorem{prop}{Proposition}
\newtheorem{thm}[prop]{Theorem}
\newtheorem{lemme}[prop]{Lemma}
\newtheorem{cor}[prop]{Corollary}

\theoremstyle{definition}

\theoremstyle{remark}

\numberwithin{equation}{section}
\title[higher torsion invariants for flat vector bundles with finite holonomy]{A remark on the higher torsion invariants \\ for flat vector bundles with finite holonomy}
\date{\today}

\author{Lie FU}
\address{Institute for Mathematics,
Astrophysics and Particle Physics,
Radboud University,
Heyendaalseweg 135, 6525 AJ,
Nijmegen, Netherland}
\email{lie.fu@math.ru.nl}

\author{Yeping ZHANG}
\address{School of Mathematics,
Korea Institute for Advanced Study,
Hoegiro 85, Dongdaemungu,
Seoul 02455, Korea}
\email{ypzhang@kias.re.kr}

\begin{document}

\begin{abstract}
We show that the Igusa-Klein topological torsion and the Bismut-Lott analytic torsion are equivalent
for any flat vector bundle whose holonomy is a finite subgroup of $\mathrm{GL}_n(\Q)$.
Our proof uses Artin's induction theorem in representation theory to reduce the problem to the special case of trivial flat line bundles,
which is a recent result of Puchol, Zhu and the second author.
The idea of using Artin's induction theorem appeared in a paper of Ohrt on the same topic,
of which our present work is an improvement. 
\end{abstract}

\maketitle

\tableofcontents

\section{Introduction}

The theory of topological torsion was developed by
Franz \cite{fr}, Reidemeister \cite{rei}, de Rham \cite{dr}, Milnor \cite{mil}, Whitehead \cite{wh} and many others.
The analytic torsion,
which is an analogue of the topological torsion,
was defined by Ray and Singer \cite{rs}.

Cheeger \cite{c-cm} and M{\"u}ller \cite{m-cm} independently proved that the topological torsion and the analytic torsion coincide for unitarily flat vector bundles.
This result is now known as the Cheeger-M{\"u}ller theorem.
Bismut, Zhang and M{\"u}ller simultaneously considered the extension of the Cheeger-M{\"u}ller theorem.
M{\"u}ller \cite{m-cm-2} extended the theorem to the unimodular case.
Bismut and Zhang \cite{bz} extended the theorem to the general case.
There are also various extensions to equivariant cases \cite{bz2,lr,lu}.

Wagoner \cite{wa} conjectured the existence of higher topological/analytic torsion invariants.
The conjectured invariant should be an invariant for pairs $(M \rightarrow S, F)$,
where $M\rightarrow S$ is a smooth fibration with compact fiber and $F$ is a flat vector bundle over $M$.
Bismut and Lott \cite{bl} confirmed the analytic side of Wagoner's conjecture
by constructing the so-called \textit{Bismut-Lott analytic torsion}.
Igusa \cite{ig} confirmed the topological side of Wagoner's conjecture
by constructing the so-called \textit{Igusa-Klein topological torsion}.
Goette, Igusa and Williams \cite{g-i-w,g-i} used Igusa-Klein topological torsion to detect the exotic smooth structure of fiber bundles.
Dwyer, Weiss and Williams \cite{dww} constructed another topological torsion.
The relation among these higher torsion invariants (in the most general case) is still unknown.

Bismut and Goette \cite{bg} showed that the Bismut-Lott torsion and the Igusa-Klein torsion are equivalent
if there exists a fiberwise Morse function $f: M \rightarrow \R$ satisfying the Morse-Smale transversality \cite{sm}.
In fact, Bismut and Goette extended the Bismut-Lott torsion to the equivariant case and proved their result in the equivariant context.
Goette \cite{g-a1,g-a2} extended the results in \cite{bg} to arbitrary fiberwise Morse functions.
There are also related works in \cite{bg2,bu}.
We refer to the survey by Goette \cite{g} for an overview on higher torsion invariants.
Goette also proposed a program extending the argument in \cite{g-a1,g-a2}
to functions with both non-degenerate critical points and birth-death critical points.

Igusa \cite{ig2} axiomatized higher torsion invariants for trivial flat line bundles.
He showed that any invariant satisfying the additivity axiom and the transfer axiom is essentially the Igusa-Klein torsion.
Badzioch, Dorabiala, Klein and Williams \cite{bdkw} showed that the Dwyer-Weiss-Williams torsion satisfies Igusa's axioms.
Using the results in \cite{ma} and \cite{pzz2},
Puchol, Zhu and the second author \cite{pzz3} showed that the Bismut-Lott torsion satisfies Igusa's axioms.
As a result, all these higher torsion invariants are equivalent for trivial flat line bundles.

As for arbitrary flat vector bundles, Ohrt \cite{ohrt} proposed a similar axiomatization approach for higher torsion invariants.
Under the assumption that the fibrations under consideration have simple fibers,
he showed that any invariant satisfying his axioms is essentially the Igusa-Klein torsion.
Puchol, Zhu and the second author \cite{pzz3} showed that the Bismut-Lott torsion also satisfies Ohrt's axioms.

The purpose of this paper is to explore the relation between the Igusa-Klein torsion and the Bismut-Lott torsion without restrictions on the fibrations.
Instead, we need to assume that the holonomy of the flat vector bundle in question lies in $\mathrm{GL}_n(\Q)$.
Our result is related to the transfer index conjecture proposed by Bunke and Gepner \cite{buge}.

\vspace{5mm}

\noindent\textbf{Acknowledgments.}
The authors are grateful to Zicheng Qian who brought Artin's induction theorem to our attention.

L. Fu is partially supported by the Agence Nationale de la Recherche (ANR) under projects ANR-20-CE40-0023 and ANR-16-CE40-0011.
Y. Zhang is supported by KIAS individual Grant MG077401 at Korea Institute for Advanced Study.

\section{Main result}

Let $M \rightarrow S$ be a smooth fibration.
Let $Z$ be the fiber.
Let $F$ be a flat vector bundle over $M$.
We assume that
\begin{itemize}
\item[-] $\pi_1(S)$ is finite;
\item[-] $Z$ is closed and oriented;
\item[-] the holonomy group of $F$ is finite.
\end{itemize}
These assumptions appear in \cite{ohrt}.
For $M \rightarrow S$ and $F$ as above,
we denote by
\begin{equation}
\tau^\mathrm{BL}(M/S,F) \in H^{\mathrm{even} \geqslant 2}(S)
\end{equation}
its Bismut-Lott analytic torsion class \cite{bl} (cf. \cite[Definition 2.1]{pzz3}),
and denote by
\begin{equation}
\tau^\mathrm{IK}(M/S,F) \in H^{\mathrm{even} \geqslant 2}(S)
\end{equation}
its Igusa-Klein topological torsion class \cite{ig}.
For $k\in\N$ and a class $a \in H^\bullet(S)$,
let $a^{[k]} \in H^k(S)$ be its component of degree $k$.
Set
\begin{align}
\label{eq-normalize}
\begin{split}
\tau^\mathrm{an}(M/S,F) & =
\sum_k \bigg\{ \frac{2^{2k}\big(k!\big)^2}{(2k+1)!} \tau^\mathrm{BL}(M/S,F) \bigg\}^{[2k]} \;,\\
\tau^\mathrm{top}(M/S,F) & =
\sum_k \bigg\{ \! - \frac{k!}{(2\pi)^k} \tau^\mathrm{IK}(M/S,F) + \frac{\zeta'(-k)\mathrm{rk}F}{2} \int_Z \mathrm{e}(TZ)\mathrm{ch}(TZ) \bigg\}^{[2k]} \;,
\end{split}
\end{align}
where $\int_Z: H^\bullet(M) \rightarrow H^\bullet(S)$ is the integration along the fiber,
$\mathrm{e}(TZ)$ (resp. $\mathrm{ch}(TZ)$) is the Euler class of the relative tangent bundle $TZ$ (resp. the Chern character of $TZ\otimes_\R\C$) ,
and $\zeta$ is the Riemann zeta function.
The first identity in \eqref{eq-normalize} is the Chern normalization introduced by Bismut and Goette \cite[Defintion 2.37]{bg}.

Let $\pi_1(M)$ be the fundamental group of $M$.
Let $\widetilde{M}$ be the universal cover of $M$,
which is canonically equipped with a right group action of $\pi_1(M)$.
For a group homomorphism $\rho: \pi_1(M) \rightarrow \mathrm{GL}_n(\C)$ with finite image,
set
\begin{equation}
F_\rho = \widetilde{M} \times_\rho \C^n \;,
\end{equation}
which is a flat vector bundle over $M$ with finite holonomy.
For convenience,
we denote
\begin{equation}
\tau^\mathrm{an/top}(M/S,\rho) = \tau^\mathrm{an/top}(M/S,F_\rho) \;.
\end{equation}

For a finite Galois extension $K/\Q$,
we denote by $\mathrm{Gal}(K/\Q)$ its Galois group.
For $g\in \mathrm{Gal}(K/\Q)$ and a group homomorphism $\rho: \pi_1(M) \rightarrow \mathrm{GL}_n(K)$,
we define
\begin{align}
\begin{split}
g.\rho: \pi_1(M) & \rightarrow \mathrm{GL}_n(K) \\
\gamma & \mapsto \big(g\big({\rho(\gamma)}_{i,j}\big)\big)_{1\leqslant i,j\leqslant n} \;,
\end{split}
\end{align}
where ${\rho(\gamma)}_{i,j}\in K$ are the entries of the matrix $\rho(\gamma) \in \mathrm{GL}_n(K)$.

\begin{thm}
\label{thm-main}
For a smooth manifold $S$ with finite fundamental group,
a smooth fibration $M \rightarrow S$ with closed oriented fiber,
a finite Galois extension $K/\Q$ and a group homomorphism $\rho: \pi_1(M) \rightarrow \mathrm{GL}_n(K)$ with finite image,
we have
\begin{equation}
\label{eq-thm-main}
\sum_{g\in\mathrm{Gal}(K/\Q)} \tau^\mathrm{an}(M/S,g.\rho) = \sum_{g\in\mathrm{Gal}(K/\Q)} \tau^\mathrm{top}(M/S,g.\rho) \;.
\end{equation}
In particular,
for a homomorphism $\rho: \pi_1(M) \rightarrow \mathrm{GL}_n(\Q)$ with finite image,
we have
\begin{equation}
\tau^\mathrm{an}(M/S,\rho) = \tau^\mathrm{top}(M/S,\rho) \;.
\end{equation}
\end{thm}

\section{A consequence of Artin's induction theorem}

Let $G$ be a finite group.
Let $R(G)$ be its representation ring with rational coefficients. In other words, as a $\Q$-vector space,
\begin{equation}
R(G) = \bigoplus_{\rho\in \operatorname{Irr}(G)} \Q \, \rho \;,
\end{equation}
where $\operatorname{Irr}(G)$ is the set of isomorphism classes of irreducible (complex) representations of $G$. The ring structure of $R(G)$ is given by tensor products.

For $\rho\in R(G)$,
let $\chi_\rho: G \rightarrow \C$ be its character.
We denote
\begin{equation}
R_\mathrm{rat}(G) = \big\{ \rho \in R(G) \;:\; \chi_\rho(g) \in \Q \text{ for any } g \in G \big\} \;.
\end{equation}
%By \cite{serre}, this coincides with the $\Q$-representation ring of $G$ with rational coefficients $K_0(\Q[G])\otimes \Q$.

For any subgroup $H \leq G$,
let $\mathbb{1} \in R(H)$ be the one-dimensional trivial representation,
let $\mathrm{Ind}^G_H\mathbb{1} \in R(G)$ be the induced representation.
Clearly, we have
\begin{equation}
\mathrm{Ind}^G_H\mathbb{1} \in R_\mathrm{rat}(G) \;.
\end{equation}

\begin{lemme}
\label{lem1}
The vector space $R_\mathrm{rat}(G)$ is spanned by $\big(\mathrm{Ind}^G_H\mathbb{1}\big)_{H \leq G}$.
\end{lemme}

The lemma above is a consequence of \cite[Exercise 13.8]{serre}.
We still give a proof for the sake of completeness.

\begin{proof}
Let $\mathcal{S}(G) \subseteq R(G)$ be the vector subspace spanned by $\big(\mathrm{Ind}^G_H\mathbb{1}\big)_{H \leq G}$.

\vspace{1mm}

\noindent Claim 1.
We have
\begin{equation}
\mathrm{Ind}_H^G \mathcal{S}(H) \subseteq \mathcal{S}(G) \hspace{2.5mm} \text{for any } H \leq G \;.
\end{equation}
This is an immediate consequence of the identity $\mathrm{Ind}_H^G\mathrm{Ind}_J^H = \mathrm{Ind}_J^G$ for $J \leq H \leq G$.

\vspace{1mm}

\noindent Claim 2.
For a surjective homomorphism $f: G \rightarrow G'$,
we have
\begin{equation}
f^*\mathcal{S}(G') \subseteq \mathcal{S}(G) \;,
\end{equation}
where $f^*: R(G') \rightarrow R(G)$ is defined by $f^*\rho = \rho \circ f$.
This is an immediate consequence of the identity $f^*\mathrm{Ind}_{H'}^{G'} = \mathrm{Ind}_H^G $ for any $H' \leq G'$ and $H = f^{-1}(H')$.

\vspace{1mm}

\noindent Claim 3.
For finite groups $G$ and $G'$,
we have
\begin{align}
\begin{split}
& R(G\times G') = R(G) \otimes R(G') \;,\\
& R_\mathrm{rat}(G\times G') = R_\mathrm{rat}(G) \otimes R_\mathrm{rat}(G') \;,\hspace{5mm}
\mathcal{S}(G\times G') \supseteq \mathcal{S}(G) \otimes \mathcal{S}(G') \;.
\end{split}
\end{align}

\vspace{1mm}

Now we are ready to prove the lemma by induction.
If $|G| = 1$, the lemma obviously holds.
Assume that
\begin{equation}
\label{eqh-pf-lem1}
\mathcal{S}(H) = R_\mathrm{rat}(H) \hspace{2.5mm} \text{for any finite group } H \text{ with } |H| < N \;.
\end{equation}
We consider a finite group $G$ with $|G| = N$.
We need to show that $R_\mathrm{rat}(G) \subseteq \mathcal{S}(G)$.
Let $K/\Q$ be a finite Galois extension such that all the representations of all the groups of order $\leqslant N$ may take values in $\mathrm{GL}_n(K)$.
There are three cases.

\vspace{1mm}

\noindent Case 1. The group $G$ is not cyclic.
%{\color{red}
%Let $\rho \in R_\mathrm{rat}(G)$. By Artin's induction theorem \cite[Theorem 26]{serre}, there exist cyclic subgroups $H_1,\cdots,H_m < G$ and $\big(\varphi_k\in R_{\mathrm{rat}}(H_k)\big)_{k=1,\cdots,m}$ such that
%\begin{equation}
%\label{eq2-pf-lem1}
%\rho = \sum_{k=1}^m \mathrm{Ind}_{H_k}^G \varphi_k \;.
%\end{equation}
%By induction hypothesis \eqref{eqh-pf-lem1}, $\varphi_k\in \mathcal{S}(H_k)$. Hence $\rho \in \mathcal{S}(G)$ by Claim 1.
%}

Let $\rho \in R_\mathrm{rat}(G)$.
We obviously have
\begin{equation}
\label{eq1-pf-lem1}
\rho = \frac{1}{[K:\Q]} \sum_{g\in\mathrm{Gal}(K/\Q)} g.\rho \;.
\end{equation}
On the other hand,
by Artin's induction theorem,
%\cite[\textsection 12.5]{serre}
there exist cyclic subgroups $H_1,\cdots,H_m \leq G$ and $\big(\varphi_k\in R(H_k)\big)_{k=1,\cdots,m}$ such that
\begin{equation}
\label{eq2-pf-lem1}
\rho = \sum_{k=1}^m \mathrm{Ind}_{H_k}^G \varphi_k \;.
\end{equation}
By \eqref{eq1-pf-lem1} and \eqref{eq2-pf-lem1},
we have
\begin{equation}
\label{eq3-pf-lem1}
\rho = \frac{1}{[K:\Q]} \sum_{k=1}^m \mathrm{Ind}_{H_k}^G \Big( \sum_{g\in\mathrm{Gal}(K/\Q)} g.\varphi_k \Big) \;.
\end{equation}
Note that $H_k$ is cyclic while $G$ is not,
by our hypothesis \eqref{eqh-pf-lem1},
we have
\begin{equation}
\label{eq4-pf-lem1}
\sum_{g\in\mathrm{Gal}(K/\Q)} g.\varphi_k \in R_\mathrm{rat}(H_k) = \mathcal{S}(H_k) \;.
\end{equation}
From \eqref{eq3-pf-lem1}, \eqref{eq4-pf-lem1} and Claim 1,
we obtain $\rho \in \mathcal{S}(G)$.
Hence $R_\mathrm{rat}(G) \subseteq \mathcal{S}(G)$.

\vspace{1mm}

\noindent Case 2. The group $G$ is cyclic, and there exist non trivial cyclic groups $G'$ and $G''$ such that $G = G' \times G''$.

By our hypothesis \eqref{eqh-pf-lem1} and Claim 3,
we have
\begin{equation}
R_\mathrm{rat}(G) = R_\mathrm{rat}(G') \otimes R_\mathrm{rat}(G'') = \mathcal{S}(G') \otimes \mathcal{S}(G'') \subseteq \mathcal{S}(G) \;.
\end{equation}

\vspace{1mm}

\noindent Case 3. The group $G$ is cyclic, and $|G| = p^r$ where $p$ is a prime number.

Let $a\in G$ be a generator of $G$.
For $k=0,\cdots,p^r-1$,
let $\varphi_k$ be the one dimensional representation defined by $\varphi_k(a) = \exp\big( 2k\pi i/p^r \big)$.
Then $\big(\varphi_k\big)_{k=0,\cdots,p^r-1}$ is a basis of $R(G)$.
Thus $R_\mathrm{rat}(G)$ is spanned by
\begin{equation}
\label{eq30-pf-lem1}
\Big(\sum_{g\in\mathrm{Gal}(K/\Q)} g.\varphi_k \Big)_{k = 0,\cdots,p^r-1} \;.
\end{equation}
If $k$ is a multiple of $p$,
then $\varphi_k: G \rightarrow \C^*$ is not injective.
There exists a surjective group homomorphism $f: G \rightarrow G'$ with $|G'|<|G|$ and $\varphi_k' \in R(G')$ such that $\varphi_k = f^*\varphi_k'$.
We have
\begin{equation}
\sum_{g\in\mathrm{Gal}(K/\Q)} g.\varphi_k = f^* \Big( \sum_{g\in\mathrm{Gal}(K/\Q)} g.\varphi_k' \Big) \in f^*R_\mathrm{rat}(G') \;.
\end{equation}
Then,
by our hypothesis \eqref{eqh-pf-lem1} and Claim 2,
we have
\begin{equation}
\sum_{g\in\mathrm{Gal}(K/\Q)} g.\varphi_k \in f^*\mathcal{S}(G') \subseteq \mathcal{S}(G) \;.
\end{equation}
If $k$ is not a multiple of $p$,
we can directly verify that
\begin{align}
\begin{split}
\frac{1}{[K:\Q]} \sum_{g\in\mathrm{Gal}(K/\Q)} g.\varphi_k
& = \frac{1}{p^{r-1}(p-1)} \Big( \sum_{k=0}^{p^r-1} \varphi_k - \sum_{k=0}^{p^{r-1}-1} \varphi_{kp} \Big) \\
& = \frac{1}{p^{r-1}(p-1)} \Big( \mathrm{Ind}^G_{\{e\}}\mathbb{1} - \mathrm{Ind}^G_{\langle a^{p^{r-1}}\rangle} \mathbb{1} \Big) \in \mathcal{S}(G) \;,
\end{split}
\end{align}
where $\{e\}$ is the trivial subgroup and $\langle a^{p^{r-1}} \rangle$ is the subgroup generated by $a^{p^{r-1}}$.
In conclusion,
each element in \eqref{eq30-pf-lem1} lies in $\mathcal{S}(G)$.
Hence $R_\mathrm{rat}(G) \subseteq \mathcal{S}(G)$.
\end{proof}

The key ingredient in the proof of Lemma \ref{lem1} is Artin's induction theorem,
which is also used in \cite[\textsection 5]{ohrt}.

\section{Proof of the main result}

Now we assume that there is a surjective group homomorphism $\mu: \pi_1(M) \rightarrow G$.
Then any linear representation of $G$ may be viewed a representation of $\pi_1(M)$.
Since
\begin{equation}
\tau^\mathrm{an/top}(M/S,\rho\oplus\rho') = \tau^\mathrm{an/top}(M/S,\rho) + \tau^\mathrm{an/top}(M/S,\rho')
\end{equation}
for $\rho,\rho'$ linear representations of $G$,
we have a $\Q$-linear map
\begin{align}
\label{eq-delta}
\begin{split}
\delta: R(G) & \rightarrow H^{\mathrm{even} \geqslant 2}(S) \\
\rho & \mapsto \tau^\mathrm{an}(M/S,\rho) - \tau^\mathrm{top}(M/S,\rho) \;.
\end{split}
\end{align}

\begin{lemme}
\label{lem2}
For any subgroup $H \leq G$,
we have $\mathrm{Ind}^G_H\mathbb{1} \in \ker \delta$.
\end{lemme}
\begin{proof}
Let $M' = \widetilde{M} \times_\mu (G/H)$,
which is a finite covering of $M$.
By \cite[Remark 5.2]{ohrt} and the induction formula for Bismut-Lott torsion \cite[Theorem 2.4]{pzz3},
we have
\begin{equation}
\label{eq1-pf-lem2}
\tau^\mathrm{an}(M'/S,\mathbb{1}) = \tau^\mathrm{an}\big(M/S,\mathrm{Ind}^G_H\mathbb{1}\big) \;.
\end{equation}
By \cite[Remark 5.2]{ohrt} and the induction formula for Igusa-Klein torsion \cite[Definition 2.1, Theorem 3.1]{ohrt},
we have
\begin{equation}
\label{eq2-pf-lem2}
\tau^\mathrm{top}(M'/S,\mathbb{1}) = \tau^\mathrm{top}\big(M/S,\mathrm{Ind}^G_H\mathbb{1}\big) \;.
\end{equation}
On the other hand,
by the higher Cheeger-M{\"u}ller/Bismut-Zhang theorem for trivial flat line bundles \cite[Theorem 0.1]{pzz3},
we have
\begin{equation}
\label{eq3-pf-lem2}
\tau^\mathrm{an}(M'/S,\mathbb{1}) = \tau^\mathrm{top}(M'/S,\mathbb{1}) \;.
\end{equation}
From \eqref{eq-delta}-\eqref{eq3-pf-lem2},
we obtain $\delta\big(\mathrm{Ind}^G_H\mathbb{1}\big) = 0$.
\end{proof}

\begin{proof}[Proof of Theorem \ref{thm-main}]
Let $\rho: \pi_1(M) \rightarrow \mathrm{GL}_n(K)$ be as in Theorem \ref{thm-main}.
Denote its image by $G$, which is a finite group by assumption.
We may and we will view $\rho$ as a representation of $G$.
We obviously have
\begin{equation}
\label{eq1-pf}
\sum_{g\in\mathrm{Gal}(K/\Q)} g.\rho \in R_\mathrm{rat}(G) \;.
\end{equation}
On the other hand,
by Lemma \ref{lem1} and Lemma  \ref{lem2},
we have
\begin{equation}
\label{eq2-pf}
R_\mathrm{rat}(G) \subseteq \ker \delta \;.
\end{equation}
From \eqref{eq-delta}, \eqref{eq1-pf}, \eqref{eq2-pf} and the linearity of $\delta$,
we obtain \eqref{eq-thm-main}.
\end{proof}

\section{Further discussions}

Let $M\rightarrow S$ and $\rho: \pi_1(M) \rightarrow \mathrm{GL}_n(K)$ be as in Theorem \ref{thm-main}.

The following result was proved in \cite[\textsection 5]{ohrt}.

\begin{lemme}
\label{lem-ohrt}
If $\tau^\mathrm{an}(M/S,\rho) = \tau^\mathrm{top}(M/S,\rho)$ for any $M\rightarrow S$ and any one dimensional representation $\rho$,
then $\tau^\mathrm{an}(M/S,\rho) = \tau^\mathrm{top}(M/S,\rho)$ for any $M\rightarrow S$ and any $\rho$.
\end{lemme}

Now we give an observation,
which is a direct consequence of Theorem \ref{thm-main} and Lemma \ref{lem-ohrt}.

\begin{cor}
If $\tau^\mathrm{an}(M/S,\rho - g.\rho) = \tau^\mathrm{top}(M/S,\rho - g.\rho)$ for any $M\rightarrow S$, any one dimensional representation $\rho$ and any $g\in\mathrm{Gal}(K/\Q)$,
then $\tau^\mathrm{an}(M/S,\rho) = \tau^\mathrm{top}(M/S,\rho)$ for any $M\rightarrow S$ and any $\rho$.
\end{cor}
\begin{proof}
By our assumption,
if $\rho$ is one dimensional,
we have
\begin{equation}
\label{eq1-pf-prop}
\tau^\mathrm{an}(M/S,\rho) - \tau^\mathrm{top}(M/S,\rho) =  \tau^\mathrm{an}(M/S,g.\rho) - \tau^\mathrm{top}(M/S,g.\rho) \hspace{4mm}
\end{equation}
for any $g\in\mathrm{Gal}(K/\Q)$.
By Theorem \ref{thm-main} and \eqref{eq1-pf-prop},
we have
\begin{align}
\begin{split}
& \tau^\mathrm{an}(M/S,\rho) - \tau^\mathrm{top}(M/S,\rho) \\
& = \frac{1}{[K:\Q]} \sum_{g\in\mathrm{Gal}(K/\Q)} \Big( \tau^\mathrm{an}(M/S,g.\rho) - \tau^\mathrm{top}(M/S,g.\rho) \Big) = 0 \;.
\end{split}
\end{align}
Hence $\tau^\mathrm{an}(M/S,\rho) = \tau^\mathrm{top}(M/S,\rho)$ for any $M\rightarrow S$ and any one dimensional $\rho$.
Now,
applying Lemma \ref{lem-ohrt},
we get $\tau^\mathrm{an}(M/S,\rho) = \tau^\mathrm{top}(M/S,\rho)$ for any $M\rightarrow S$ and any $\rho$.
\end{proof}

\bibliographystyle{amsplain}
\bibliography{ArtinInd}

\end{document}